\documentclass[11pt,a4paper,dvipsnames]{article}

\usepackage[top=24mm, bottom=30mm, right=38mm, left=38mm]{geometry}
\usepackage[T1]{fontenc}
\usepackage[utf8]{inputenc}
\usepackage[english]{babel}

\usepackage{newtxtext}
\usepackage{microtype}
\usepackage{amsmath,amsthm,mathtools,empheq}
\usepackage{newtxmath}
\allowdisplaybreaks[4]
\usepackage{bm}
\usepackage{needspace}
\usepackage{stmaryrd}
\usepackage{cases}
\numberwithin{equation}{section}
\usepackage{titling}
\usepackage{xcolor}
\usepackage{algorithm}
\usepackage{algpseudocode}
\floatname{algorithm}{\sffamily Algorithm}
\usepackage{enumitem}
\usepackage[numbers,sort]{natbib}
\usepackage[bf,sf]{titlesec}
\usepackage{caption}
\usepackage{aliascnt}

\definecolor{ForestGreen}{rgb}{0.13, 0.55, 0.13}
\definecolor{MidnightBlue}{rgb}{0.098, 0.098, 0.439}

\usepackage{xurl}
\usepackage{hyperref}
\hypersetup{
    pdftitle      = {An optimal first-order method for smooth and strongly convex composite optimization and its stationary limit},
    pdfauthor     = {Manu Upadhyaya, Daniel Berg Thomsen, Aymeric Dieuleveut, and Adrien B. Taylor},
    colorlinks    = true,
    urlcolor      = MidnightBlue,
    linkcolor     = MidnightBlue,
    citecolor     = ForestGreen,
}
\newcommand{\doi}[1]{\href{https://doi.org/#1}{doi: \nolinkurl{#1}}}
\newcommand{\arxiv}[2][]{\href{https://arxiv.org/abs/#2}{arXiv:\nolinkurl{#2}\if\relax\detokenize{#1}\relax\else{} [#1]\fi}}
\usepackage[capitalize,nameinlink]{cleveref}
\renewcommand{\thealgorithm}{\arabic{algorithm}}

\makeatletter
\renewcommand{\theHALG@line}{\thealgorithm.\arabic{ALG@line}}
\makeatother
\newcommand{\algorithmkeyword}[1]{\textsf{\textbf{#1}}}
\algrenewcommand\algorithmicrequire{\algorithmkeyword{Input:}}
\algnewcommand\algorithmicinitialization{\algorithmkeyword{Initialization:}}
\algnewcommand\Initialization{\item[\algorithmicinitialization]}
\algrenewcommand\alglinenumber[1]{\footnotesize\color{black!55}#1}
\newcommand{\algnamefont}{\ttfamily}
\newcommand{\ProxITEM}{{\hyperref[alg:prox-item]{\algnamefont Prox-ITEM}}}
\newcommand{\refProxITEM}{\ProxITEM\ (\Cref{alg:prox-item})}
\newcommand{\ProxTMM}{{\hyperref[alg:prox-tmm]{\algnamefont Prox-TMM}}}
\newcommand{\refProxTMM}{\ProxTMM\ (\Cref{alg:prox-tmm})}

\let\mathbf\bm

\newcommand{\set}[1]{\mathord{\left.\{#1\}\right.}}
\newcommand{\Bigset}[1]{\mathord{\left\{#1\right\}}}
\newcommand{\p}[1]{\mathord{(#1)}}
\newcommand{\Bigp}[1]{\mathord{\left(#1\right)}}

\newcommand{\Bignorm}[1]{\left\lVert#1\right\rVert}
\newcommand{\norm}[1]{\lVert#1\rVert}
\newcommand{\Biginner}[2]{\left\langle #1, #2 \right\rangle}
\newcommand{\inner}[2]{\langle #1, #2 \rangle}
\newcommand{\xmiddle}[1]{\;\middle#1\;}

\newcommand{\Prox}{\operatorname{Prox}}
\DeclareMathOperator*{\minimize}{minimize}
\DeclareMathOperator*{\argmin}{argmin}
\DeclareMathOperator{\dom}{dom}

\renewcommand{\leq}{\leqslant}
\renewcommand{\geq}{\geqslant}

\newcommand{\reals}{\mathbb{R}}
\newcommand{\N}{\mathbb{N}}
\newcommand{\naturals}{\mathbb{N}_0}
\newcommand{\calH}{\mathcal{H}}

\newif\ifsiopt
\sioptfalse
\newcommand{\sioptcite}[2]{\ifsiopt\cite{#2}\else\cite{#1}\fi}

\providecommand{\MSC}[1]{\textbf{\textsf{Mathematics subject classification 2020.}} #1}
\providecommand{\keywords}[1]{\textbf{\textsf{Keywords.}} #1}

\makeatletter
\renewenvironment{abstract}{%
  \if@twocolumn
    \section*{{\sffamily\bfseries \abstractname}}%
  \else
    \small
    \begin{center}%
      {\sffamily\bfseries \abstractname\vspace{-.5em}\vspace{\z@}}%
    \end{center}%
    \begin{quote}%
    \setlength{\parindent}{0pt}%
    \noindent\ignorespaces
  \fi}
  {\if@twocolumn\else\end{quote}\fi}
\makeatother

\newtheoremstyle{myStyle1}%
  {0.3cm}%
  {0.3cm}%
  {\itshape}%
  {}%
  {}%
  {:}%
  {.5em}%
  {%
    \thmname{\normalfont\sffamily\bfseries #1 }%
    \thmnumber{\normalfont\sffamily\bfseries #2}%
    \thmnote{\normalfont\sffamily\bfseries \ (#3)}%
  }

\newtheoremstyle{myStyle2}%
  {0.3cm}%
  {0.3cm}%
  {}%
  {}%
  {}%
  {:}%
  {.5em}%
  {%
    \thmname{\normalfont\sffamily\bfseries #1 }%
    \thmnumber{\normalfont\sffamily\bfseries #2}%
    \thmnote{\normalfont\sffamily\bfseries \ (#3)}%
  }
\theoremstyle{myStyle1}
\newtheorem{theorem}{Theorem}[section]
\newaliascnt{proposition}{theorem}
\newtheorem{proposition}[proposition]{Proposition}
\aliascntresetthe{proposition}
\newaliascnt{lemma}{theorem}
\newtheorem{lemma}[lemma]{Lemma}
\aliascntresetthe{lemma}
\newaliascnt{assumption}{theorem}
\newtheorem{assumption}[assumption]{Assumption}
\aliascntresetthe{assumption}
\newlist{enumeratass}{enumerate}{1}
\setlist[enumeratass]{nosep, label={(\roman*)}, ref=\theassumption(\roman*)}
\newaliascnt{definition}{theorem}
\newtheorem{definition}[definition]{Definition}
\aliascntresetthe{definition}
\theoremstyle{myStyle2}
\newaliascnt{remark}{theorem}
\newtheorem{remark}[remark]{Remark}
\aliascntresetthe{remark}
\newlist{enumeratremark}{enumerate}{1}
\newcommand{\remarklistbreak}{\leavevmode\par}
\setlist[enumeratremark]{label={(\roman*)}, ref=\theremark(\roman*), before=\remarklistbreak}

\AddToHook{env/theorem/begin}{\crefalias{section}{theorem}}
\AddToHook{env/proposition/begin}{\crefalias{theorem}{proposition}}
\AddToHook{env/lemma/begin}{\crefalias{theorem}{lemma}}
\AddToHook{env/assumption/begin}{\crefalias{theorem}{assumption}}
\AddToHook{env/definition/begin}{\crefalias{theorem}{definition}}
\AddToHook{env/remark/begin}{\crefalias{theorem}{remark}}

\crefname{theorem}{theorem}{theorems}
\Crefname{theorem}{Theorem}{Theorems}
\crefname{proposition}{proposition}{propositions}
\Crefname{proposition}{Proposition}{Propositions}
\crefname{lemma}{lemma}{lemmas}
\Crefname{lemma}{Lemma}{Lemmas}
\crefname{assumption}{assumption}{assumptions}
\Crefname{assumption}{Assumption}{Assumptions}
\crefname{enumeratassi}{assumption}{assumptions}
\Crefname{enumeratassi}{Assumption}{Assumptions}
\crefname{definition}{definition}{definitions}
\Crefname{definition}{Definition}{Definitions}
\crefname{remark}{remark}{remarks}
\Crefname{remark}{Remark}{Remarks}
\crefname{algorithm}{algorithm}{algorithms}
\Crefname{algorithm}{Algorithm}{Algorithms}
\crefname{enumeratremarki}{remark}{remarks}
\Crefname{enumeratremarki}{Remark}{Remarks}

\title{\sffamily\bfseries An optimal first-order method for smooth and strongly convex composite optimization \\ and its stationary limit \\[2ex]}
\author{%
    \begin{tabular}{@{}c@{\hspace{2.5em}}c@{}}
        Manu Upadhyaya$^{1,2}$ &
        Daniel Berg Thomsen$^{1,2}$ \\
        {\small\texttt{\href{mailto:manu.upadhyaya@inria.fr}{manu.upadhyaya@inria.fr}}} &
        {\small\texttt{\href{mailto:daniel.berg-thomsen@inria.fr}{daniel.berg-thomsen@inria.fr}}} \\
        \\[-0.2ex]
        Aymeric Dieuleveut$^{2}$ &
        Adrien B. Taylor$^{1}$ \\
        {\small\texttt{\href{mailto:aymeric.dieuleveut@polytechnique.edu}{aymeric.dieuleveut@polytechnique.edu}}} &
        {\small\texttt{\href{mailto:adrien.taylor@inria.fr}{adrien.taylor@inria.fr}}}
    \end{tabular}
    \\[5.5ex]
    {\small $^{1}$Inria, D.I. ENS, CNRS, PSL Research University, Paris, France}\\[0.5ex]
    {\small $^{2}$CMAP, \'Ecole Polytechnique, Institut Polytechnique de Paris, Palaiseau, France}
}
\date{}

\begin{document}

\maketitle


\begin{abstract}
    We introduce \ProxITEM{}, an optimal proximal-gradient method for minimizing \(f+g\), where \(f\) is smooth and strongly convex, and \(g\) is convex, proper, and lower semicontinuous. In the smooth case \(g=0\), \ProxITEM{} reduces to the information-theoretic exact method (ITEM). We prove an exact distance-to-solution bound for \ProxITEM{} with the same distance-convergence rate as ITEM, and show that this rate is minimax optimal among span-based first-order methods using the same number of gradient-oracle calls for \(f\) and an arbitrary number of proximal-oracle calls for \(g\). We also identify the stationary limit of \ProxITEM{}, denoted \ProxTMM{}, which gives a proximal extension of the triple momentum method (TMM) to the composite setting and achieves the corresponding TMM distance-convergence rate.
\end{abstract}


\MSC{
    65K05, 
    68Q25, 
    90C25, 
    93D30  
}


\section{Introduction}\label{sec:introduction}

Composite optimization problems arise when the objective naturally separates into one smooth term and one term that carries additional structure, such as constraints, sparsity-promoting regularization, or other nonsmooth penalties.
In particular, we consider
\begin{align}\label{eq:problem_class}
    \minimize_{x \in \calH}\;
    f\p{x} + g\p{x},
\end{align}
and throughout this paper, we make the following assumptions.
\begin{assumption}
    \phantomsection\label{ass:main}
    \begin{enumeratass}
        \item\label{ass:main:hilbert}
        \(\p{\calH,\inner{\cdot}{\cdot}}\) is a real Hilbert space where \(\norm{\cdot}\) is the canonical norm.
        \item\label{ass:main:f}
        The function \(f:\calH\to\reals\) is \(\mu\)-strongly convex and \(L\)-smooth, where \(0<\mu<L<+\infty\).
        \item\label{ass:main:g}
        The function \(g:\calH\to\reals\cup\{+\infty\}\) is convex, proper, and lower semicontinuous.
    \end{enumeratass}
\end{assumption}

The main question addressed here is whether the convergence rate of
the infor\-mation-theoretic exact method (ITEM) \cite[Algorithm~1]{taylor2023optimal_gradient} for smooth and strongly convex
minimization can be
retained in the composite setting \eqref{eq:problem_class} by a first-order
method that accesses \(f\) through gradient-oracle calls and \(g\) through
proximal-oracle calls.
We answer this question by introducing \refProxITEM{}, a proximal extension of
ITEM for the composite setting, and proving the exact bound
\begingroup
\ifsiopt
\setlength{\abovedisplayskip}{0.35\baselineskip}
\setlength{\abovedisplayshortskip}{0.35\baselineskip}
\setlength{\belowdisplayskip}{0.35\baselineskip}
\setlength{\belowdisplayshortskip}{0.35\baselineskip}
\fi
\[
    \p{\forall k\in\naturals}\qquad
    \norm{z^k-x^{\star}}^2
    \leq
    \frac{1}{1+qA_k}\norm{x^0-x^{\star}}^2,
\]
\endgroup
where \(x^{\star}\in\argmin_{x\in\calH}\p{f\p{x} + g\p{x}}\).

\begin{algorithm}[t]
    \caption{\ProxITEM}
    \label{alg:prox-item}
    \begin{algorithmic}[1]
        \Require \(f\in\mathcal{F}_{\mu,L}\p{\mathcal{H}}\) with \(0<\mu<L<+\infty\), \(g\in\mathcal{F}_{0,\infty}\p{\mathcal{H}}\), and \(x^0 \in \mathcal{H}\).
        \Initialization \(q = \mu/L\), \(A_0 = 0\), and \(z^0 = x^0\).
        \item[\algorithmkeyword{For}] \(k=0,1,2,\dots\)
            \begingroup
            \setlength{\abovedisplayskip}{0.35\baselineskip}
            \setlength{\abovedisplayshortskip}{0.35\baselineskip}
            \setlength{\belowdisplayskip}{0.35\baselineskip}
            \setlength{\belowdisplayshortskip}{0.35\baselineskip}
            \begin{align}
                A_{k+1}
                &=
                \frac{\p{1+q}A_k+2\p{1+\sqrt{\p{1+A_k}\p{1+qA_k}}}}{\p{1-q}^2}
                \label{eq:prox-item:A-recursion}\\
                \beta_k
                &=
                \frac{A_k}{\p{1-q}A_{k+1}}
                \label{eq:prox-item:beta}\\
                \delta_k
                &=
                \sqrt{\frac{A_{k+1}}{1+qA_{k+1}}}
                \label{eq:prox-item:delta}\\
                y^k
                &=
                \p{1-\beta_k}z^k
                + \beta_k x^k
                \label{eq:prox-item:extrapolation}\\
                \bar z^{k+1}
                &=
                \p{1-q\delta_k}z^k
                + q\delta_k y^k
                - \frac{\delta_k}{L}\nabla f\p{y^k}
                \label{eq:prox-item:prox-point}\\
                z^{k+1}
                &=
                \Prox^{\delta_k/L}_{g}\p{\bar z^{k+1}}
                \label{eq:prox-item:z-update}\\
                x^{k+1}
                &=
                y^k
                - \frac{1}{L}\nabla f\p{y^k}
                - \frac{1}{\delta_k}\p{\bar z^{k+1}-z^{k+1}}
                \label{eq:prox-item:x-update}
            \end{align}
            \endgroup
    \end{algorithmic}
\end{algorithm}

In particular, the factor \(\p{1+qA_k}^{-1}\) is the same as for ITEM \cite[Theorem~3]{taylor2023optimal_gradient}, and this rate is known to be a complexity lower bound in the smooth and strongly convex setting \cite{drori2021exact}.
Since that smooth subclass is contained in the composite model \eqref{eq:problem_class}, this gives a matching lower bound for any
span-based method using \(k\) gradient-oracle calls for \(f\) and an arbitrary number of proximal-oracle calls for \(g\).
Thus, \ProxITEM{} is minimax optimal with respect to this algorithmic class and performance measure.

The coefficients of \ProxITEM{} have a stationary limit.
In particular, \(\beta_k\to\allowbreak\p{1-\sqrt{q}}/\p{1+\sqrt{q}}\) and \(\delta_k\to\allowbreak1/\sqrt{q}\) as \(k\to\infty\).
Substituting these limiting values into \ProxITEM{} gives the stationary recursion defining \refProxTMM{}\@.
The method \ProxTMM{} is the composite counterpart of the triple momentum method (TMM) \cite{van_scoy_freeman_lynch_2018_tmm}.
We prove that its main sequence satisfies
\[
    \norm{z^k-x^{\star}}
    \in
    \mathcal{O}\p{\p{1-\sqrt{q}}^k}
    \quad\text{as}\quad k\to\infty.
\]
Both analyses are based on conceptually simple Lyapunov arguments. In particular, when
\(g=0\), the additional nonsmooth terms vanish, and the Lyapunov
function of \ProxITEM{} reduces to the Lyapunov function of ITEM
\cite[Section~2.1]{taylor2023optimal_gradient}.
Moreover, the Lyapunov function used for \ProxTMM{} is obtained as the stationary limit of the \ProxITEM{} Lyapunov function.

\ifsiopt\else
\begin{algorithm}[t]
    \caption{\ProxTMM}
    \label{alg:prox-tmm}
    \begin{algorithmic}[1]
        \Require \(f\in\mathcal{F}_{\mu,L}\p{\mathcal{H}}\) with \(0<\mu<L<+\infty\), \(g\in\mathcal{F}_{0,\infty}\p{\mathcal{H}}\), and \(x^0 \in \mathcal{H}\).
        \Initialization \(q = \mu/L\) and \(z^0 = x^0\).
        \item[\algorithmkeyword{For}] \(k=0,1,2,\dots\)
            \begingroup
            \setlength{\abovedisplayskip}{0.35\baselineskip}
            \setlength{\abovedisplayshortskip}{0.35\baselineskip}
            \setlength{\belowdisplayskip}{0.35\baselineskip}
            \setlength{\belowdisplayshortskip}{0.35\baselineskip}
            \begin{align}
                y^k
                &=
                \frac{2\sqrt{q}}{1+\sqrt{q}}z^k
                + \frac{1-\sqrt{q}}{1+\sqrt{q}}x^k \\
                \bar z^{k+1}
                &=
                \p{1-\sqrt{q}}z^k
                + \sqrt{q}y^k
                - \frac{1}{\sqrt{q}L}\nabla f\p{y^k} \\
                z^{k+1}
                &=
                \Prox^{1/\p{\sqrt{q}L}}_{g}\p{\bar z^{k+1}}
                \label{eq:prox-tmm:z-update}\\
                x^{k+1}
                &=
                y^k
                - \frac{1}{L}\nabla f\p{y^k}
                - \sqrt{q}\p{\bar z^{k+1}-z^{k+1}}
            \end{align}
            \endgroup
    \end{algorithmic}
\end{algorithm}

\pagebreak[4]
\fi

A useful point of comparison is OptISTA \cite{jang2025optista}, which is optimal in terms of function-value suboptimality for composite problems where \(f\) is smooth and convex but not necessarily strongly convex.
The relevant distinction is that the Lyapunov function used to analyze OptISTA is history-dependent, i.e., it involves past iterates, whereas our Lyapunov functions depend only on current-iterate information, leading to substantially shorter proofs.

A surprising feature of \ProxITEM{} is that the proximal operation should not be
applied to the usual point \(y^{k}-\frac{1}{L}\nabla f(y^k)\).
Rather, the proximal operation should be applied to the \(\p{\bar{z}^{k}}_{k\in\N}\)-sequence, closer in spirit to the approaches of \cite{auslender_teboulle_2006_interior_gradient,tseng_2008_accelerated_proximal_gradient}.

\ifsiopt\begin{algorithm}[t]
    \caption{\ProxTMM}
    \label{alg:prox-tmm}
    \begin{algorithmic}[1]
        \Require \(f\in\mathcal{F}_{\mu,L}\p{\mathcal{H}}\) with \(0<\mu<L<+\infty\), \(g\in\mathcal{F}_{0,\infty}\p{\mathcal{H}}\), and \(x^0 \in \mathcal{H}\).
        \Initialization \(q = \mu/L\) and \(z^0 = x^0\).
        \item[\algorithmkeyword{For}] \(k=0,1,2,\dots\)
            \begingroup
            \setlength{\abovedisplayskip}{0.35\baselineskip}
            \setlength{\abovedisplayshortskip}{0.35\baselineskip}
            \setlength{\belowdisplayskip}{0.35\baselineskip}
            \setlength{\belowdisplayshortskip}{0.35\baselineskip}
            \begin{align}
                y^k
                &=
                \frac{2\sqrt{q}}{1+\sqrt{q}}z^k
                + \frac{1-\sqrt{q}}{1+\sqrt{q}}x^k \\
                \bar z^{k+1}
                &=
                \p{1-\sqrt{q}}z^k
                + \sqrt{q}y^k
                - \frac{1}{\sqrt{q}L}\nabla f\p{y^k} \\
                z^{k+1}
                &=
                \Prox^{1/\p{\sqrt{q}L}}_{g}\p{\bar z^{k+1}}
                \label{eq:prox-tmm:z-update}\\
                x^{k+1}
                &=
                y^k
                - \frac{1}{L}\nabla f\p{y^k}
                - \sqrt{q}\p{\bar z^{k+1}-z^{k+1}}
            \end{align}
            \endgroup
    \end{algorithmic}
\end{algorithm}
\fi


\subsection{Related work}\label{subsec:related_work}

The problem class \eqref{eq:problem_class} is the standard setting for the proximal-gradient method, more generally known as forward-backward splitting:
the smooth term is accessed through its gradient, whereas the nonsmooth term is accessed through its proximal operator.
The method itself goes back to the operator-splitting literature
\cite{lions_mercier_1979_splitting,passty_1979_ergodic} and is now a basic tool
in convex optimization and signal recovery
\cite{beck_2017_first_order_methods,combettes_wajs_2005_signal_recovery}.
Moreover, exact worst-case analyses of the proximal-gradient method are
available in the composite setting \cite{taylor_hendrickx_glineur_2018_proximal_gradient}.
Acceleration in first-order methods traces back to the fast-gradient method and
complexity lower bounds \cite{nesterov1983method,nesterov2018lectures}. Its
composite counterparts include accelerated proximal-gradient and FISTA-type
methods \cite{beck_teboulle_2009_fista,nesterov2013_gradient_composite,tseng_2008_accelerated_proximal_gradient},
as well as related one-projection and universal variants
\cite{auslender_teboulle_2006_interior_gradient,gasnikov_nesterov_2018_universal_stochastic_composite}.
These methods provide the classical accelerated rates for function-value
suboptimality; for additional background on acceleration, see, e.g.,
\cite{daspremont_scieur_taylor_2021_acceleration_methods}.

Technically, our analysis draws on exact worst-case analysis and computer-aided
design of first-order methods. The performance-estimation problem (PEP) framework reduces
worst-case analysis of fixed-step first-order methods to finite-dimensional
semidefinite programs \cite{drori_teboulle_2014_pep}. Interpolation theory
subsequently gave exact PEP formulations for smooth and (strongly) convex
minimization \cite{taylor2016smoothstronglyconvex} and for composite
settings \cite{taylor_hendrickx_glineur_2017_composite_pep}.
The PEP methodology also led to optimized smooth convex methods such as OGM
\cite{drori2017exact_information,kim_fessler_2016_ogm,kim_fessler_2018_generalizing_ogm,kim_fessler_2021_ogm_g}.
It has also been extended using branch-and-bound techniques for nonconvex PEP design problems \cite{gupta_van_parys_ryu_2024_bnb_pep}
and via numerical design of optimized algorithms using local methods
\cite{kamri_hendrickx_glineur_2025_numerical_design}. In the composite convex
setting, OptISTA \cite{jang2025optista} gives an optimal method for
function-value suboptimality, and recent reduction approaches transfer optimized
smooth methods to the composite setting, including proximal versions of OGM and
Silver stepsize schedules
\sioptcite{altschuler_parrilo_2025_silver_smooth,altschuler_parrilo_2025_stepsize_hedging_i,bok_altschuler_2025_composite_reduction,bok_altschuler_2025_silver_proxgd}{altschuler_parrilo_2025_stepsize_hedging_i,altschuler_parrilo_2025_silver_smooth,bok_altschuler_2025_silver_proxgd,bok_altschuler_2025_composite_reduction}.
Related long-step schedules for smooth gradient descent, analyzed through
multi-step rather than per-iteration descent arguments, were developed in \cite{grimmer_2024_long_steps}.
For quadratic objectives, the link between optimal linear rates and minimax
polynomials is classical, going back to Chebyshev semi-iterative methods
\cite{golub_varga_1961_chebyshev_i,golub_varga_1961_chebyshev_ii,varga2000matrix_iterative_analysis}.
ITEM extends this optimality picture from quadratic objectives to the full class
of smooth and strongly convex functions, with an exact rate and a matching lower
bound \cite{drori2021exact,taylor2023optimal_gradient}.

Previous attempts to extend ITEM to the composite setting include \cite{florea_2024_adaptive_first_order,ushiyama_2025_sqrt2_accelerated_fista}.
However, these methods do not achieve the same rate as ITEM\@.
Moreover, projected counterparts of unconstrained strongly convex methods have also been designed to
preserve linear rate bounds under quadratic Lyapunov certificates
\cite{li_lestas_nagahara_2025_projected_same_rates}, but are not applicable in our setting.

Finally, there is a related but distinct literature on designing optimal and accelerated proximal-point, fixed-point, and monotone-inclusion methods. Halpern iterations and their
optimized variants provide tight residual guarantees for fixed-point problems
involving nonexpansive or Lipschitz mappings
\cite{bravo_cominetti_lee_2026_minimax_halpern,contreras_cominetti_2023_optimal_error_bounds,halpern1967fixed_points,lieder2021halpern,park_ryu_2022_exact_fixed_point}.
Accelerated proximal-point methods, including classical resolvent-based schemes
and APPM, are designed for oracle models for convex minimization or maximally
monotone operators \cite{barre_taylor_bach_2023_inexact_proximal,guler_1992_new_proximal_point,kim2021accelerated_proximal_point}.
Those results concern oracles and performance measures different from those of
the forward-backward model studied here.


\subsection{Organization}\label{subsec:organization}

The remainder of this manuscript is organized as follows. In
\Cref{subsec:notation_preliminaries}, we collect the notation and present basic results used in the
sequel. \Cref{sec:main_results} presents the main results: the convergence
guarantee for \ProxITEM, its matching optimality result over span-based
first-order methods, and the asymptotic rate of the stationary method
\ProxTMM{}\@. \Cref{sec:auxiliary_estimates_convergence_analyses}
collects the auxiliary estimates used in the proofs: first the Lyapunov
analysis for \ProxITEM, and then the limiting-coefficient and Lyapunov
analysis for \ProxTMM{}\@.
\Cref{sec:proof_main_prox_item,sec:proof_main_prox_item_is_optimal,sec:proof_main_prox_tmm}
prove the three main theorems, respectively. Finally, \Cref{sec:conclusion}
contains concluding remarks.


\subsection{Notation and preliminaries}\label{subsec:notation_preliminaries}

Let \(\naturals\) denote the set of nonnegative integers, \(\N\) denote the
set of positive integers, \(\llbracket n,m \rrbracket = \set{l \in \mathbb{Z} \mid n \leq l \leq m}\)
denote the set of integers between \(n\) and \(m\), inclusive, where
\(\mathbb{Z}\) denotes the set of integers, \(\reals\) denotes the set of real numbers, \(\reals_+\)
denotes the set of nonnegative real numbers, and \(\reals_{++}\) denotes the set of positive real numbers.

Let \(h:\calH\to\reals\cup\{+\infty\}\) and \(\mu,L\in\reals_+\).
The function \(h\) is said to be \emph{proper} if \(\dom h\neq\emptyset\), where \(\dom h=\set{x\in\calH \mid h\p{x}<+\infty}\) is called the \emph{effective domain} of \(h\).
It is \emph{lower semicontinuous} if \(\liminf_{y\to x} h\p{y} \geq h\p{x}\) for each \(x\in\calH\).
It is \emph{convex} if \(h\p{\p{1-\lambda}x+\lambda y}\leq \p{1-\lambda}h\p{x}+\lambda h\p{y}\) for each \(x,y\in\calH\) and \(\lambda\in[0,1]\).
It is \(\mu\)-\emph{strongly convex} if \(h\) is proper and \(h-\frac{\mu}{2}\norm{\cdot}^2\) is convex.
It is \(L\)-\emph{smooth} if \(h\) is Fr{\'e}chet differentiable and \(\nabla h:\calH\to\calH\) is \(L\)-Lipschitz continuous, i.e., \(\Bignorm{\nabla h\p{x}-\nabla h\p{y}}\leq\allowbreak L\Bignorm{x-y}\) for each \(x,y\in\calH\).

\begin{definition}
    \phantomsection\label{def:calF}
    Let \(0 \leq \mu \leq L \leq +\infty\), with \(\mu<+\infty\).
    We let \(\mathcal{F}_{\mu,L}\p{\mathcal{H}}\) denote the class of all
    proper and lower semicontinuous functions
    \(h:\calH\rightarrow\reals\cup\{+\infty\}\) that are
    \(\mu\)-strongly convex and, if \(L < +\infty\), \(L\)-smooth.
\end{definition}

Under \Cref{ass:main}, the optimization problem \eqref{eq:problem_class} has a unique solution \cite[Corollaries~9.4 and~11.17]{bauschke_combettes_2017_convex_analysis_monotone}.

Let \(h:\calH\to\reals\cup\{+\infty\}\).
If \(h\) is proper, its \emph{subdifferential} is the set-valued operator \(\partial h:\calH\to2^{\calH}\) defined by
\(\partial h \p{x}
= \Bigset{u \in \calH \xmiddle| h\p{y} \geq h\p{x} + \Biginner{u}{y-x}
\text{ for each } y \in \calH}\)
for each \(x \in \calH\).
The effective domain of \(\partial h\) is \(\dom\partial h = \set{x\in\calH \mid \partial h\p{x}\neq\emptyset}\), and \(\dom\partial h\subseteq\dom h\); see \cite[Proposition 16.4]{bauschke_combettes_2017_convex_analysis_monotone}.
If \(h\) is proper, convex, and lower semicontinuous and \(\gamma\in\reals_{++}\), then its
\emph{proximal operator} is the single-valued operator \(\Prox^{\gamma}_{h}:\calH\to\calH\) given by \(\Prox^{\gamma}_{h}\p{x} =\allowbreak \argmin_{z\in\calH}\p{h\p{z} +\allowbreak \frac{1}{2\gamma}\Bignorm{x-z}^2}\) for each \(x\in\calH\), with \(h\p{\Prox^{\gamma}_{h}\p{x}} < + \infty\); see~\cite[Proposition 12.15]{bauschke_combettes_2017_convex_analysis_monotone}.
In particular, by \cite[\mbox{Propositions}~16.44 and~16.6]{bauschke_combettes_2017_convex_analysis_monotone},
\begin{equation}\label{eq:prox_characterization}
    \p{\forall x,p\in\calH}\quad
    \left[
    \begin{gathered}[c]
        p = \Prox^{\gamma}_{h}\p{x}
        \, \Leftrightarrow \,
        \gamma^{-1}\p{x-p} \in \partial h\p{p}
    \end{gathered}
    \right].
\end{equation}

\begin{proposition}[{\cite[Theorem 4]{taylor2016smoothstronglyconvex}}]\label{prp:smooth_strongly_convex_interpolation}
    {\sloppy
    Let \(0 \leq \mu < L\leq +\infty\), and let \(\{\p{x_{i},F_{i},u_{i}}\}_{i\in \mathcal{J}}\) be a finite family of triplets in \(\calH\times\allowbreak\reals\times\allowbreak\calH\) indexed by \(\mathcal{J}\).
    Then the following are equivalent:
    \par}
    \begin{enumerate}[label={(\roman*)}]
        \item There exists \(h\in\mathcal{F}_{\mu,L}\p{\calH}\) such that
        \(h\p{x_{i}} = F_i\) and \(u_{i} \in \partial h\p{x_{i}}\) for each
        \(i\in\mathcal{J}\).
        \item\label{prp:smooth_strongly_convex_interpolation:ii}
        It holds that \(F_{i} \geq F_{j} + \inner{u_{j}}{x_{i}-x_{j}} + \frac{\mu}{2}\norm{x_{i}-x_{j}}^{2} + \frac{1}{2\p{L-\mu}}\norm{u_{i}-u_{j}-\mu\p{x_{i}-x_{j}}}^2\) for each \(i,j\in\mathcal{J}\), where \(\frac{1}{2\p{L-\mu}}\) is interpreted as \(0\) in the case \(L = +\infty\).
    \end{enumerate}
\end{proposition}

\Cref{prp:smooth_strongly_convex_interpolation} motivates the following
notation, and we refer to the underlying inequalities as interpolation
inequalities: under \Cref{ass:main}, for arbitrary \(x,y\in\calH\), set
\begin{equation}\label{eq:smooth_interpolation_inequality}
    \begin{aligned}
        \mathcal{I}_{f}\p{x,y}
        ={}&{}
        f\p{x}-f\p{y}
        -\inner{\nabla f\p{y}}{x-y}
        -\frac{\mu}{2}\norm{x-y}^{2} \\
        &{}-
        \frac{1}{2\p{L-\mu}}
        \norm{\nabla f\p{x}-\nabla f\p{y}-\mu\p{x-y}}^{2}
        \in \reals_{+}.
    \end{aligned}
\end{equation}
For \(x,y,s\in\calH\), define the symbol
\begin{equation}\label{eq:convex_interpolation_inequality}
    \mathcal{I}_{g}\p{x,y,s}
    =
    g\p{x}-g\p{y}
    -\inner{s}{x-y}.
\end{equation}
If \(x,y\in\dom\partial g\) and \(s\in\partial g\p{y}\), then
\begin{align}
    \mathcal{I}_{g}\p{x,y,s}
    &\in \reals_{+}.
    \label{eq:convex_interpolation_inequality:nonnegative}
\end{align}


\section{Main results}\label{sec:main_results}

We first state the exact distance guarantee for \ProxITEM{}, then its matching optimality result over a class of span-based methods, and finally the asymptotic rate of \ProxTMM{}\@.

\begin{theorem}
    \phantomsection\label{thm:main:prox_item}
    Suppose that \Cref{ass:main} holds.
    Let~\(((x^k,\allowbreak y^k,\allowbreak z^k,\allowbreak A_k))_{k\in\naturals}\) be generated by \refProxITEM, let \(x^{\star}\in\argmin_{x\in\calH}\allowbreak\p{f\p{x} + g\p{x}}\), and let \(q=\mu/L\).
    Then
    \begin{align}\label{eq:main:prox_item}
        \p{\forall k\in\naturals}\quad
        \norm{z^k-x^{\star}}^{2}\leq \frac{1}{1+qA_{k}}\norm{x^0-x^{\star}}^{2}.
    \end{align}
\end{theorem}
The proof of \Cref{thm:main:prox_item} is given in
\Cref{sec:proof_main_prox_item}.

\begin{remark}
    \begin{enumeratremark}
        \item When \(g=0\), the rate in \eqref{eq:main:prox_item} matches the
        ITEM rate in \cite[Theorem~3]{taylor2023optimal_gradient}.
        \item The following two simple problem instances attain equality
        in \eqref{eq:main:prox_item}; they are the smooth examples from
        \cite[Section~2.3]{taylor2023optimal_gradient} with \(g=0\).
        \begin{enumerate}[label=(\alph*)]
            \item If \(f=\frac{L}{2}\norm{\cdot}^{2}\) and \(g=0\), then
            \(x^{\star}=0\) and
            \(z^{k}=\p{-1}^{k}\p{1+qA_k}^{-1/2}x^0\) for each \(k\in\naturals\).
            \item If \(f=\frac{\mu}{2}\norm{\cdot}^{2}\) and \(g=0\), then
            \(x^{\star}=0\) and
            \(z^{k}=\p{1+qA_k}^{-1/2}x^0\) for each \(k\in\naturals\).
        \end{enumerate}
    \end{enumeratremark}
\end{remark}
We next state the optimality of \ProxITEM{} under \Cref{ass:main}, relative
to a span-based class \(\mathcal{A}_{k,\mu,L}\) of first-order methods with \(k\)
gradient-oracle calls for \(f\) and an arbitrary number of proximal-oracle calls for \(g\), where \(k\in\naturals\) is fixed.
A method \(\mathsf A\in\mathcal{A}_{k,\mu,L}\) in this class takes an initial point
\(u^0\in\calH\) and returns
\(\mathsf A\p{f,g,u^0}=u^k\)\nolinebreak[4]\ according to the following rule:
\begin{equation}\label{eq:span-based-proximal-method-update}
    \left[
    \begin{aligned}
        &v^{0} = u^{0},\\
        &m\in\naturals \text{ such that } m\geq k,\\
        &\theta \in \set{0,1}^{\llbracket0,m-1\rrbracket} \text{ such that } \textstyle\sum_{\ell=0}^{m-1}\theta_{\ell} = k, \\
        &s^{\ell} =
        \begin{cases}
            \nabla f \p{v^{\ell}} &\text{if } \theta_{\ell} = 1,\\
            \vcenter{\hbox{\(\underbracket{\gamma_{\ell}^{-1} \p{v^{\ell} - \Prox_g^{\gamma_{\ell}}\p{v^{\ell}}}}_{\in \partial g \p{\Prox_g^{\gamma_{\ell}}\p{v^{\ell}}}}\)}} & \vcenter{\hbox{\(\text{if } \theta_{\ell} = 0 \text{ for some }\gamma_{\ell}\in\reals_{++},\)}}
        \end{cases}
        \quad \forall \ell \in \llbracket0,m-1\rrbracket, \\
        &v^{\ell} \in u^{0} + \operatorname{span}\set{s^{i} \mid i \in \llbracket 0,\ell-1 \rrbracket }\quad \forall \ell \in \llbracket 1,m-1 \rrbracket, \\
        &u^{k} \in u^{0} + \operatorname{span}\set{s^{i} \mid i \in \llbracket 0,m-1 \rrbracket }
    \end{aligned}
    \right.
\end{equation}
Here, we use the convention \(\operatorname{span}\emptyset=\set{0}\).
For \(\mathsf A\in\mathcal{A}_{k,\mu,L}\), we define the worst-case relative squared
distance to the solution as
\begin{equation}\label{eq:fixed-step-proximal-method-risk}
    \mathcal{R}_{\mu,L}\p{\mathsf A}
    =
    \sup
    \Bigset{
        \frac{\norm{\mathsf A\p{f,g,u^0}-u^{\star}}^2}{\norm{u^0-u^{\star}}^2}
        \xmiddle|
        \begin{gathered}
            \calH\text{ is a real Hilbert space,}\\
            f\in\mathcal{F}_{\mu,L}\p{\calH},\\
            g\in\mathcal{F}_{0,\infty}\p{\calH},\\
            u^{\star}\in\argmin_{u\in\calH}\p{f\p{u}+g\p{u}},\\
            u^0\in\calH,\quad u^0\neq u^{\star}
        \end{gathered}
    }.
\end{equation}

\begin{theorem}
    \phantomsection\label{thm:main:prox_item_is_optimal}
    Let~\(0<\mu<L<+\infty\).
    Set \(q=\mu/L\), denote by \(\p{A_k}_{k\in\naturals}\) the scalar
    sequence generated by \refProxITEM, and define \(\mathcal{R}_{\mu,L}\) by
    \eqref{eq:fixed-step-proximal-method-risk}. Then
    \begin{equation}\label{eq:main:prox_item_is_optimal}
        \p{\forall k\in\naturals}\quad
        \inf_{\mathsf A\in\mathcal{A}_{k,\mu,L}}
        \mathcal{R}_{\mu,L}\p{\mathsf A}
        =
        \frac{1}{1+qA_k}.
    \end{equation}
    For each \(k\in\naturals\), the infimum in
    \eqref{eq:main:prox_item_is_optimal} is attained by \refProxITEM{} after
    \(k\) iterations, with input \(x^0=u^0\) and output \(u^k=z^k\).
\end{theorem}
The proof of \Cref{thm:main:prox_item_is_optimal} is given in
\Cref{sec:proof_main_prox_item_is_optimal}.

\begin{theorem}
    \phantomsection\label{thm:main:prox_tmm}
    Suppose that \Cref{ass:main} holds.
    Let~\(((x^k,\allowbreak y^k,\allowbreak z^k,\allowbreak \bar z^{k+1}))_{k\in\naturals}\) be generated by \refProxTMM, let \(x^{\star}\in\argmin_{x\in\calH}\allowbreak\p{f\p{x} + g\p{x}}\), and let \(q=\mu/L\).
    Then
    \begin{equation*}
        \p{\forall k\in\N}\quad
        \norm{z^k-x^{\star}}^2
        \leq
        c\p{1-\sqrt{q}}^{2\p{k-1}},
    \end{equation*}
    where
    \begin{equation*}
        c
        =
        \frac{1-q}{\mu}\mathcal{I}_f\p{y^0,x^{\star}}
        +\frac{q}{\mu}\mathcal{I}_g\p{x^{\star},z^1,s_g^1}
        +\frac{1}{2\mu L}\norm{s_g^1-s_g^{\star}}^2
        +\norm{z^1-x^{\star}}^2,
    \end{equation*}
    with
    \begin{equation*}
        s_g^1
        =
        \sqrt{q}L\p{\bar z^1-z^1}
        \in
        \partial g\p{z^1}
        \quad\text{and}\quad
        s_g^{\star}
        =
        -\nabla f\p{x^{\star}}
        \in
        \partial g\p{x^{\star}}.
    \end{equation*}
\end{theorem}
The proof of \Cref{thm:main:prox_tmm} is given in
\Cref{sec:proof_main_prox_tmm}.

\begin{remark}
    When \(g=0\), the rate in \Cref{thm:main:prox_tmm}\ifsiopt\hfil\penalty-500\hfilneg\fi\space
    matches the TMM rate in \mbox{\cite[Theorem~1]{van_scoy_freeman_lynch_2018_tmm}}.
\end{remark}


\section{Auxiliary estimates for the convergence analyses}\label{sec:auxiliary_estimates_convergence_analyses}

This section develops the Lyapunov analysis underlying the convergence results for \ProxITEM{} and \ProxTMM{}\@.
The two subsections treat the nonstationary method and its stationary limit, respectively.


\subsection{Lyapunov analysis for \texorpdfstring{\ProxITEM}{Prox-ITEM}}\label{subsec:prox-item-lyapunov-analysis}

\begin{definition}[Lyapunov function]\label{def:prox-item-lyapunov}
    \emergencystretch=2em\relax
    Suppose that \Cref{ass:main} holds.
    Let~\(((x^k,\allowbreak y^k,\allowbreak z^k,\allowbreak \bar z^{k+1},\allowbreak A_k,\allowbreak \delta_k))_{k\in\naturals}\) be generated by \refProxITEM{}\@.
    Let~\(x^{\star}\in\argmin_{x\in\calH}\allowbreak\p{f\p{x} + g\p{x}}\).
    Define
    \begin{equation}\label{eq:prox-item:lyapunov}
        \begin{aligned}
            \p{\forall k\in\naturals}\quad
            \mathcal{V}_k
            ={}&{}
            \p{1-q}A_k\mathcal{I}_f\p{y^{k-1},x^{\star}}
            +qA_k\mathcal{I}_g\p{x^{\star},z^k,s_g^k} \\
            &{}+
            \frac{A_k}{2L}\norm{s_g^k-s_g^{\star}}^2
            +
            \p{L+\mu A_k}\norm{z^k-x^{\star}}^2,
        \end{aligned}
    \end{equation}
    where
    \begin{align}
        \p{\forall k\in\N}\quad
        s_g^k
        &=
        L\delta_{k-1}^{-1}\p{\bar z^k-z^k}
        \in
        \partial g\p{z^k},
        \label{eq:prox-item:sampled-subgradient}\\
        s_{g}^{\star}
        &=
        - \nabla f\p{x^{\star}}
        \in
        \partial g\p{x^{\star}}.
        \label{eq:prox-item:optimal-subgradient}
    \end{align}
    The inclusion in \eqref{eq:prox-item:sampled-subgradient} follows from \eqref{eq:prox-item:z-update} and \eqref{eq:prox_characterization}.
    The quantities \(\mathcal{I}_f\) and \(\mathcal{I}_g\) in \eqref{eq:prox-item:lyapunov} are defined in \eqref{eq:smooth_interpolation_inequality} and \eqref{eq:convex_interpolation_inequality}, respectively.
    In what follows, we use the conventions \(y^{-1}=y^0\), \(s_g^0=s_g^1\), \(0\cdot\p{+\infty}=0\), and \(0\cdot\p{-\infty}=0\).
\end{definition}

\begin{remark}
    \begin{enumeratremark}
        \item When \(g=0\), the proximal operator reduces to the identity mapping; therefore, both the \(\mathcal{I}_g\)-term and the subgradient-norm term in \eqref{eq:prox-item:lyapunov} vanish.
        The resulting Lyapunov function is the ITEM Lyapunov function (4) in \cite[Section~2.1]{taylor2023optimal_gradient}, up to notation, with \(\mathcal{I}_f\p{y^{k-1},x^{\star}}\) corresponding to their \(\psi_{k-1}\).
        \item\label{rem:prox-item-initial-lyapunov} With the conventions in \Cref{def:prox-item-lyapunov}, \(A_0=0\), and \(z^0=x^0\), \eqref{eq:prox-item:lyapunov} gives
        \(\mathcal{V}_0=L\norm{x^0-x^{\star}}^2\).
    \end{enumeratremark}
\end{remark}

\begin{lemma}\label{lem:prox-item-nonnegative-coefficients}
    Let~\(\p{A_k,\delta_k}_{k\in\naturals}\)\ifsiopt\hfil\penalty-500\hfilneg\fi\space
    be generated by \refProxITEM, \mbox{and define}
    \begin{equation}\label{eq:prox-item:sigma}
        \p{\forall k\in\naturals}\quad
        \sigma_k
        =
        \sqrt{\p{1+A_k}\p{1+qA_k}}.
    \end{equation}
    Then
    \begin{empheq}[left={\p{\forall k\in\naturals}\quad\empheqlbrack}]{align}
        A_k
        &\geq 0,
        \label{eq:prox-item:nonnegative-A}\\
        A_{k+1}
        &> 0,
        \label{eq:prox-item:positive-next-A}\\
        \delta_{k}
        &> 0,
        \label{eq:prox-item:positive-delta}\\
        A_{k+1}-A_k
        &\geq 0,
        \label{eq:prox-item:nonnegative-A-increment}\\
        \p{1+q}A_{k+1}-A_k-\sigma_k+1
        &\geq 0,
        \label{eq:prox-item:nonnegative-g-coefficient-two}\\
        \sigma_k-1-qA_k
        &\geq 0.
        \label{eq:prox-item:nonnegative-sigma-minus-one}
    \end{empheq}
\end{lemma}

\begin{proof}
    Since \refProxITEM{} sets \(q=\mu/L\) with \(0<\mu<L\), we have \(q\in\p{0,1}\).
    Moreover, \(A_0=0\).
    We first prove \eqref{eq:prox-item:nonnegative-A} and \eqref{eq:prox-item:positive-next-A}.
    It suffices to show that \(A_k\geq 0\) implies \(A_{k+1}>0\).
    Under this assumption, \eqref{eq:prox-item:sigma} gives
    \begin{equation*}
        \sigma_k
        =
        \sqrt{\p{1+A_k}\p{1+qA_k}}
        \geq 1,
    \end{equation*}
    and hence \eqref{eq:prox-item:A-recursion} gives
    \begin{equation*}
        A_{k+1}
        =
        \frac{\p{1+q}A_k+2\p{1+\sigma_k}}{\p{1-q}^2}
        \geq
        \frac{4}{\p{1-q}^2}
        > 0.
    \end{equation*}
    Therefore, by induction, \(A_k\geq 0\) and \(A_{k+1}>0\) for each \(k\in\naturals\).
    Then \eqref{eq:prox-item:delta} gives
    \begin{equation*}
        \delta_k
        =
        \sqrt{\frac{A_{k+1}}{1+qA_{k+1}}}
        >
        0,
    \end{equation*}
    which proves \eqref{eq:prox-item:positive-delta}.
    For each \(k\in\naturals\), substituting \eqref{eq:prox-item:A-recursion} gives
    \begin{equation*}
        \p{1-q}^{2}\p{A_{k+1}-A_k}
        =
        q\p{3-q}A_k+2\p{1+\sigma_k}
        \geq 0.
    \end{equation*}
    Since \(\p{1-q}^{2}>0\), this proves \eqref{eq:prox-item:nonnegative-A-increment}.
    Another substitution of \eqref{eq:prox-item:A-recursion} gives
    \begin{equation*}
        \begin{aligned}
            \p{1-q}^{2}
            \p{\p{1+q}A_{k+1}-A_k-\sigma_k+1}
            =
            4qA_k+\p{3+q^2}+\p{1+4q-q^2}\sigma_k
            \geq 0.
        \end{aligned}
    \end{equation*}
    Since \(\p{1-q}^{2}>0\), this proves \eqref{eq:prox-item:nonnegative-g-coefficient-two}.
    Finally, applying \eqref{eq:prox-item:sigma} with \(A_k\geq 0\) gives \(\sigma_k^2-\p{1+qA_k}^2=\p{1-q}A_k\p{1+qA_k}\geq 0\), and hence \(\sigma_k\geq 1+qA_k\), which proves \eqref{eq:prox-item:nonnegative-sigma-minus-one}.
\end{proof}

\begin{remark}
    \begin{enumeratremark}
        \item By \eqref{eq:prox-item:positive-delta}, the proximal parameter \(\delta_k/L\) in \eqref{eq:prox-item:z-update} is positive for each \(k\in\naturals\), and the denominator \(\delta_{k-1}\) in \eqref{eq:prox-item:sampled-subgradient} is positive for each \(k\in\N\).
        Thus, \eqref{eq:prox-item:z-update} and \eqref{eq:prox-item:sampled-subgradient} are well defined.
        \item The conventions in \Cref{def:prox-item-lyapunov}, together with \eqref{eq:smooth_interpolation_inequality}, \eqref{eq:convex_interpolation_inequality:nonnegative}, and \eqref{eq:prox-item:nonnegative-A}, imply that each term on the right-hand side of \eqref{eq:prox-item:lyapunov} is nonnegative. Thus, \(\mathcal{V}_k\in\reals_+\) for each \(k\in\naturals\).
        \item\label{rem:prox-item-distance-estimate} In particular, retaining only the last term in \eqref{eq:prox-item:lyapunov} gives
        \begin{equation*}
            \p{\forall k\in\naturals}\quad
            \p{L+\mu A_k}\norm{z^k-x^{\star}}^2
            \leq
            \mathcal{V}_k.
        \end{equation*}
    \end{enumeratremark}
\end{remark}

\begin{lemma}[Lyapunov equality]\label{lem:prox-item-one-step-lyapunov}
    \sloppy
    Suppose that \Cref{ass:main} holds.
    Let~\(\p{\mathcal{V}_k}_{k\in\naturals}\) be as in \Cref{def:prox-item-lyapunov}.
    Then
    \begin{equation}\label{eq:prox-item:one-step-lyapunov}
        \p{\forall k\in\naturals}\quad
        \mathcal{V}_{k+1}
        =
        \mathcal{V}_k-\mathcal{R}_k,
    \end{equation}
    where
    \begin{align}
        \p{\forall k\in\naturals}\quad
        \mathcal{R}_k
        ={}&{}
        \frac{A_{k+1}-A_k}{2L}\norm{s_g^{k+1}-s_g^{\star}}^2
        +\frac{A_k}{2L}\norm{s_g^k-s_g^{k+1}}^2 \notag\\
        {}&{}+
        \p{1-q}\p{A_{k+1}-A_k}\mathcal{I}_f\p{x^{\star},y^k}
        +\p{1-q}A_k\mathcal{I}_f\p{y^{k-1},y^k} \notag\\
        {}&{}+
        \p{A_{k+1}-A_k}\mathcal{I}_g\p{x^{\star},z^{k+1},s_g^{k+1}} \notag\\
        {}&{}+
        \p{\sigma_k-1-qA_k}\mathcal{I}_g\p{z^k,x^{\star},s_g^{\star}} \notag\\
        {}&{}+
        \p{\p{1+q}A_{k+1}-A_k-\sigma_k+1}
        \mathcal{I}_g\p{z^{k+1},x^{\star},s_g^{\star}} \notag\\
        {}&{}+
        \p{\sigma_k-1}\mathcal{I}_g\p{z^{k+1},z^k,s_g^k}.
        \label{eq:prox-item:one-step-lyapunov-remainder}
    \end{align}
    Moreover, \(\mathcal{R}_k\in\reals_+\) for each \(k\in\naturals\).
\end{lemma}

\begin{proof}
    First, note that, for each \(k\in\naturals\), \eqref{eq:smooth_interpolation_inequality}, \eqref{eq:convex_interpolation_inequality:nonnegative}, \Cref{lem:prox-item-nonnegative-coefficients}, and the conventions in \Cref{def:prox-item-lyapunov} give
    \begin{align*}
        \mathcal{R}_k
        ={}&{}
        \underbracket{\frac{A_{k+1}-A_k}{2L}}_{\geq 0}
        \underbracket{\norm{s_g^{k+1}-s_g^{\star}}^2}_{\geq 0}
        +\underbracket{\frac{A_k}{2L}}_{\geq 0}
        \underbracket{\norm{s_g^k-s_g^{k+1}}^2}_{\geq 0} \notag\\
        {}&{}+
        \underbracket{\p{1-q}\p{A_{k+1}-A_k}}_{\geq 0}
        \underbracket{\mathcal{I}_f\p{x^{\star},y^k}}_{\geq 0}
        +\underbracket{\p{1-q}A_k}_{\geq 0}
        \underbracket{\mathcal{I}_f\p{y^{k-1},y^k}}_{\geq 0} \notag\\
        {}&{}+
        \underbracket{\p{A_{k+1}-A_k}}_{\geq 0}
        \underbracket{\mathcal{I}_g\p{x^{\star},z^{k+1},s_g^{k+1}}}_{\geq 0} \notag\\
        {}&{}+
        \underbracket{\p{\sigma_k-1-qA_k}\mathcal{I}_g\p{z^k,x^{\star},s_g^{\star}}}_{\mathclap{\substack{=0\text{ if } k=0\\ \geq 0\text{ if } k\in\N}}} \notag\\
        {}&{}+
        \underbracket{\p{\p{1+q}A_{k+1}-A_k-\sigma_k+1}}_{\geq 0}
        \underbracket{\mathcal{I}_g\p{z^{k+1},x^{\star},s_g^{\star}}}_{\geq 0} \notag\\
        {}&{}+
        \underbracket{\p{\sigma_k-1}\mathcal{I}_g\p{z^{k+1},z^k,s_g^k}}_{\mathclap{\substack{=0\text{ if } k=0\\ \geq 0\text{ if } k\in\N}}}
        \in\reals_+.
    \end{align*}

    Next, define
    \begin{equation}\label{eq:prox-item:one-step-lyapunov-slack-definition}
        \p{\forall k\in\naturals}\quad
        \mathcal{S}_k
        =
        \mathcal{V}_{k+1}-\mathcal{V}_k+\mathcal{R}_k;
    \end{equation}
    it remains to show that \(\mathcal{S}_k=0\).

    Substituting \eqref{eq:prox-item:lyapunov} and \eqref{eq:prox-item:one-step-lyapunov-remainder} into \(\mathcal{S}_k\), and using \eqref{eq:smooth_interpolation_inequality}, the coefficients of the function values in the \(\mathcal{I}_f\)-terms are
    \begin{align*}
        f\p{y^{k-1}}:
        {}&{}
        -\p{1-q}A_k+\p{1-q}A_k=0, \\
        f\p{y^k}:
        {}&{}
        \p{1-q}A_{k+1}-\p{1-q}\p{A_{k+1}-A_k}-\p{1-q}A_k=0, \\
        f\p{x^{\star}}:
        {}&{}
        -\p{1-q}A_{k+1}+\p{1-q}A_k+\p{1-q}\p{A_{k+1}-A_k}=0.
    \end{align*}
    Similarly, using \eqref{eq:convex_interpolation_inequality}, the coefficients of the function values in the \(\mathcal{I}_g\)-terms are
    \begin{align*}
        g\p{x^{\star}}:
        {}&{}
        qA_{k+1}-qA_k+\p{A_{k+1}-A_k}
        -\p{\sigma_k-1-qA_k} \\
        &{}-
        \p{\p{1+q}A_{k+1}-A_k-\sigma_k+1}=0, \\
        g\p{z^{k+1}}:
        {}&{}
        -qA_{k+1}-\p{A_{k+1}-A_k}
        +\p{\p{1+q}A_{k+1}-A_k-\sigma_k+1} \\
        &{}+
        \p{\sigma_k-1}=0, \\
        g\p{z^k}:
        {}&{}
        qA_k+\p{\sigma_k-1-qA_k}-\p{\sigma_k-1}=0.
    \end{align*}
    Thus, all function values in \(\mathcal{S}_k\) cancel.

    Expanding the remaining parts of the interpolation inequalities gives
    \begin{align}
        \mathcal{S}_k
        ={}&{}
        \frac{A_{k+1}-A_k}{2L}\norm{s_g^{k+1}-s_g^{\star}}^2
        +\frac{A_k}{2L}\norm{s_g^k-s_g^{k+1}}^2 \notag\\
        {}&{}
        -\p{1-q}\p{A_{k+1}-A_k}\p{
            \inner{\nabla f\p{y^k}}{x^{\star}-y^k}
            +\frac{\mu}{2}\norm{x^{\star}-y^k}^{2}
        } \notag\\
        {}&{}
        -\frac{\p{1-q}\p{A_{k+1}-A_k}}{2\p{L-\mu}}
        \norm{\nabla f\p{x^{\star}}-\nabla f\p{y^k}-\mu\p{x^{\star}-y^k}}^2 \notag\\
        {}&{}
        -\p{1-q}A_k\p{
            \inner{\nabla f\p{y^k}}{y^{k-1}-y^k}
            +\frac{\mu}{2}\norm{y^{k-1}-y^k}^{2}
        } \notag\\
        {}&{}
        -\frac{\p{1-q}A_k}{2\p{L-\mu}}
        \norm{\nabla f\p{y^{k-1}}-\nabla f\p{y^k}-\mu\p{y^{k-1}-y^k}}^2 \notag\\
        {}&{}
        -\p{1-q}A_{k+1}\p{
            \inner{\nabla f\p{x^{\star}}}{y^k-x^{\star}}
            +\frac{\mu}{2}\norm{y^k-x^{\star}}^{2}
        } \notag\\
        {}&{}
        -\frac{\p{1-q}A_{k+1}}{2\p{L-\mu}}
        \norm{\nabla f\p{y^k}-\nabla f\p{x^{\star}}-\mu\p{y^k-x^{\star}}}^2 \notag\\
        {}&{}
        +\p{1-q}A_k\p{
            \inner{\nabla f\p{x^{\star}}}{y^{k-1}-x^{\star}}
            +\frac{\mu}{2}\norm{y^{k-1}-x^{\star}}^{2}
        } \notag\\
        {}&{}
        +\frac{\p{1-q}A_k}{2\p{L-\mu}}
        \norm{\nabla f\p{y^{k-1}}-\nabla f\p{x^{\star}}-\mu\p{y^{k-1}-x^{\star}}}^2 \notag\\
        {}&{}
        -\p{\p{1+q}A_{k+1}-A_k}\inner{s_g^{k+1}}{x^{\star}-z^{k+1}}
        +qA_k\inner{s_g^k}{x^{\star}-z^k} \notag\\
        {}&{}
        -\p{\p{1+q}A_{k+1}-A_k-\sigma_k+1}
        \inner{s_g^{\star}}{z^{k+1}-x^{\star}} \notag\\
        {}&{}
        -\p{\sigma_k-1-qA_k}\inner{s_g^{\star}}{z^k-x^{\star}}
        -\p{\sigma_k-1}\inner{s_g^k}{z^{k+1}-z^k} \notag\\
        {}&{}
        -\frac{A_k}{2L}\norm{s_g^k-s_g^{\star}}^2
        +\frac{A_{k+1}}{2L}\norm{s_g^{k+1}-s_g^{\star}}^2 \notag\\
        {}&{}
        +\p{L+\mu A_{k+1}}\norm{z^{k+1}-x^{\star}}^2
        -\p{L+\mu A_k}\norm{z^k-x^{\star}}^2.
        \label{eq:prox-item:one-step-lyapunov-raw-slack}
    \end{align}

    Using \(q=\mu/L\), \eqref{eq:prox-item:extrapolation},
    \eqref{eq:prox-item:prox-point}, and \eqref{eq:prox-item:x-update},
    together with \eqref{eq:prox-item:sampled-subgradient},
    \eqref{eq:prox-item:optimal-subgradient}, and the conventions in
    \Cref{def:prox-item-lyapunov}, and leaving \(A_{k+1}\), \(\beta_k\),
    \(\delta_k\), and \(\sigma_k\) unevaluated, we obtain from
    \eqref{eq:prox-item:one-step-lyapunov-raw-slack} after direct
    simplification\footnote{\raggedright Symbolic verification notebooks are available at
    \href{https://github.com/PerformanceEstimation/ProxITEM}{\nolinkurl{https://github.com/PerformanceEstimation/}}\allowbreak\href{https://github.com/PerformanceEstimation/ProxITEM}{\mbox{\nolinkurl{ProxITEM}}}.\par}
    \begin{align*}
        \mathcal{S}_k
        &={}
        p_0
        \inner{z^k-x^k}{
            L\delta_k^{-1}\p{z^k-z^{k+1}}
            +qL\p{x^{\star}-z^k}
            +s_g^{\star}-s_g^{k+1}
        }
        \\
        &\quad{}+
        \inner{z^k-z^{k+1}}{\eta^k},
    \end{align*}
    where
    \begin{align*}
        \eta^k
        &={}
        p_1\delta_k^{-2}
        \p{
            L\p{x^{\star}-z^k}
            +\delta_k\p{s_g^{k+1}-s_g^{\star}}
        }
        +p_2\delta_k^{-1}\p{s_g^k-s_g^{\star}} \\
        &\quad{}+
        p_3\delta_k^{-2}L\p{z^k+z^{k+1}-2x^{\star}},
    \end{align*}
    and
    \begin{align*}
        p_0
        &={}
        A_k-\p{1-q}A_{k+1}\beta_k, \\
        p_1
        &={}
        \delta_k\p{A_k-\p{1+q}A_{k+1}}
        +2A_{k+1}, \\
        p_2
        &={}
        \delta_k\p{\sigma_k-1}
        -A_k, \\
        p_3
        &={}
        A_{k+1}
        -\p{1+qA_{k+1}}\delta_k^2.
    \end{align*}

    \begin{sloppypar}
        Each polynomial coefficient vanishes for the parameters generated by
        \refProxITEM{}.
        For \(p_0=0\), use the definition of \(\beta_k\) in
        \eqref{eq:prox-item:beta}.
        For \(p_1=p_2=0\), use
        \(\delta_k^2=A_{k+1}\p{1+qA_{k+1}}^{-1}\) from
        \eqref{eq:prox-item:delta}, the recursion for \(A_{k+1}\) from
        \eqref{eq:prox-item:A-recursion}, and
        \(\sigma_k^2=\p{1+A_k}\p{1+qA_k}\) from
        \eqref{eq:prox-item:sigma}.
        For \(p_3=0\), only the substitution
        \(\delta_k^2=A_{k+1}\p{1+qA_{k+1}}^{-1}\) from
        \eqref{eq:prox-item:delta} is needed.
        Therefore,
        \begin{equation}
            \p{\forall k\in\naturals}\quad
            \mathcal{S}_k=0.
            \label{eq:prox-item:one-step-lyapunov-slack-zero}
        \end{equation}
    \end{sloppypar}
    Combining \eqref{eq:prox-item:one-step-lyapunov-slack-definition} with
    \eqref{eq:prox-item:one-step-lyapunov-slack-zero} gives
    \eqref{eq:prox-item:one-step-lyapunov}.
\end{proof}


\subsection{Limiting-coefficient and Lyapunov analysis for \texorpdfstring{\ProxTMM}{Prox-TMM}}\label{subsec:prox-tmm-limiting-coefficient-lyapunov-analysis}

\begin{lemma}\label{lem:prox-tmm-limiting-coefficients}
    Suppose that \Cref{ass:main} holds.
    Let~\(((A_k,\allowbreak \beta_k,\allowbreak \delta_k))_{k\in\naturals}\) be generated by \refProxITEM, and let \(\p{\sigma_k}_{k\in\naturals}\) be defined by \eqref{eq:prox-item:sigma}.
    Then
    \begin{align}
        A_k
        &\xrightarrow[k\to\infty]{}
        \infty,
        \label{eq:prox-tmm:limit-A-infinity}\\
        \frac{A_k}{A_{k+1}}
        &\xrightarrow[k\to\infty]{}
        \p{1-\sqrt{q}}^{2},
        \label{eq:prox-tmm:limit-A-ratio}\\
        \beta_k
        &\xrightarrow[k\to\infty]{}
        \frac{1-\sqrt{q}}{1+\sqrt{q}},
        \label{eq:prox-tmm:limit-beta}\\
        \delta_k
        &\xrightarrow[k\to\infty]{}
        \frac{1}{\sqrt{q}},
        \label{eq:prox-tmm:limit-delta}\\
        \frac{\sigma_k}{A_k}
        &\xrightarrow[k\to\infty]{}
        \sqrt{q},
        \label{eq:prox-tmm:limit-sigma-self-ratio}\\
        \frac{\sigma_k}{A_{k+1}}
        &\xrightarrow[k\to\infty]{}
        \sqrt{q}\p{1-\sqrt{q}}^{2}.
        \label{eq:prox-tmm:limit-sigma-ratio}
    \end{align}
\end{lemma}

\begin{proof}
    Since \(q=\mu/L\), \Cref{ass:main} gives \(q\in\p{0,1}\).
    By \Cref{lem:prox-item-nonnegative-coefficients}, \(A_k\geq 0\) and
    \(\sigma_k\geq 1\) for each \(k\in\naturals\).
    Subtracting \(A_k\) from \eqref{eq:prox-item:A-recursion} gives
    \begin{equation*}
        A_{k+1}-A_k
        =
        \frac{q\p{3-q}A_k+2\p{1+\sigma_k}}{\p{1-q}^2}
        \geq
        \frac{4}{\p{1-q}^2}.
    \end{equation*}
    Hence \(A_k\to\infty\), which proves
    \eqref{eq:prox-tmm:limit-A-infinity}.

    For \(k\in\N\), using \eqref{eq:prox-item:sigma},
    \begin{equation*}
        \frac{\sigma_k}{A_k}
        =
        \sqrt{\Bigp{1+\dfrac{1}{A_k}}\Bigp{q+\dfrac{1}{A_k}}}
        \xrightarrow[k\to\infty]{}
        \sqrt{q}.
    \end{equation*}
    This proves \eqref{eq:prox-tmm:limit-sigma-self-ratio}.
    For \(k\in\N\), dividing \eqref{eq:prox-item:A-recursion} by \(A_k\)
    gives
    \begin{equation*}
        \frac{A_{k+1}}{A_k}
        =
        \frac{1+q+2\p{1+\sigma_k}/A_k}{\p{1-q}^2}
        \xrightarrow[k\to\infty]{}
        \frac{\p{1+\sqrt{q}}^2}{\p{1-q}^2}
        =
        \frac{1}{\p{1-\sqrt{q}}^2}.
    \end{equation*}
    Taking reciprocals gives \eqref{eq:prox-tmm:limit-A-ratio}.

    The remaining limits follow directly from the definitions:
    \begin{align*}
        \beta_k
        &=
        \frac{1}{1-q}\frac{A_k}{A_{k+1}}
        \xrightarrow[k\to\infty]{}
        \frac{1-\sqrt{q}}{1+\sqrt{q}}, \\
        \delta_k
        &=
        \sqrt{\frac{1}{q+1/A_{k+1}}}
        \xrightarrow[k\to\infty]{}
        \frac{1}{\sqrt{q}}, \\
        \frac{\sigma_k}{A_{k+1}}
        &=
        \frac{\sigma_k}{A_k}\frac{A_k}{A_{k+1}}
        \xrightarrow[k\to\infty]{}
        \sqrt{q}\p{1-\sqrt{q}}^{2}.
    \end{align*}
    These are \eqref{eq:prox-tmm:limit-beta},
    \eqref{eq:prox-tmm:limit-delta}, and
    \eqref{eq:prox-tmm:limit-sigma-ratio}.
\end{proof}

\begin{remark}
    The limits \eqref{eq:prox-tmm:limit-beta} and \eqref{eq:prox-tmm:limit-delta} show that \refProxTMM{} is the stationary limit of \refProxITEM{}.
    Indeed, substituting the limiting values of \(\beta_k\) and \(\delta_k\) into \eqref{eq:prox-item:extrapolation}, \eqref{eq:prox-item:prox-point}, \eqref{eq:prox-item:z-update}, and \eqref{eq:prox-item:x-update} gives the update rules of \refProxTMM{}\@.
\end{remark}

\begin{definition}[Lyapunov function]\label{def:prox-tmm-lyapunov}
    {\sloppy
    Suppose that \Cref{ass:main} holds.
    Let the iterates be generated by \refProxTMM{}\@.
    Let~\(x^{\star}\in\argmin_{x\in\calH}\allowbreak\p{f\p{x} + g\p{x}}\), and let \(q=\mu/L\).
    Define
    \par}
    \begin{equation}\label{eq:prox-tmm:lyapunov}
        \begin{aligned}
            \p{\forall k\in\N}\quad
            \mathcal{V}_k^{\infty}
            ={}&{}
            \p{1-q}\mathcal{I}_f\p{y^{k-1},x^{\star}}
            +q\mathcal{I}_g\p{x^{\star},z^k,s_g^k} \\
            &{}+
            \frac{1}{2L}\norm{s_g^k-s_g^{\star}}^2
            +
            \mu\norm{z^k-x^{\star}}^2,
        \end{aligned}
    \end{equation}
    where
    \begin{equation}\label{eq:prox-tmm:sampled-subgradient}
        \p{\forall k\in\N}\quad
        s_g^k
        =
        \sqrt{q}L\p{\bar z^k-z^k}
        \in
        \partial g\p{z^k}.
    \end{equation}
    The inclusion in \eqref{eq:prox-tmm:sampled-subgradient} follows from \eqref{eq:prox-tmm:z-update} and \eqref{eq:prox_characterization}, and \(s_g^{\star}\) is defined in \eqref{eq:prox-item:optimal-subgradient}.
\end{definition}

\begin{remark}
    \begin{enumeratremark}
        \item\label{rem:prox-tmm-lyapunov:stationary-limit} The formula for the
        stationary quantity \(\mathcal{V}_k^{\infty}\) in
        \eqref{eq:prox-tmm:lyapunov} is obtained coefficientwise by dividing
        the nonstationary quantity \(\mathcal{V}_k\) in
        \eqref{eq:prox-item:lyapunov} by \(A_k\) for \(k\in\N\) and then
        replacing only the scalar coefficients, not the iterates or residual
        arguments, by their limits as \(k\to\infty\), using
        \Cref{lem:prox-tmm-limiting-coefficients}.
        Indeed,
        \begin{equation*}
            \begin{aligned}
                \frac{\mathcal{V}_k}{A_k}
                ={}&{}
                \p{1-q}\mathcal{I}_f\p{y^{k-1},x^{\star}}
                +q\mathcal{I}_g\p{x^{\star},z^{k},s_g^{k}} \\
                &{}+
                \frac{1}{2L}\norm{s_g^k-s_g^{\star}}^2
                +
                \underbracket{\frac{L+\mu A_k}{A_k}}_{\xrightarrow[k\to\infty]{}\mu}\norm{z^{k}-x^{\star}}^2.
            \end{aligned}
        \end{equation*}
        \item\label{rem:prox-tmm-lyapunov:real-valued} The interpolation inequalities
        \eqref{eq:smooth_interpolation_inequality} and
        \eqref{eq:convex_interpolation_inequality:nonnegative}, the inclusions
        \eqref{eq:prox-tmm:sampled-subgradient} and
        \eqref{eq:prox-item:optimal-subgradient}, and \(q\in\p{0,1}\) imply
        that each term on the right-hand side of
        \eqref{eq:prox-tmm:lyapunov} is nonnegative. Thus,
        \(\mathcal{V}_k^{\infty}\in\reals_+\) for each \(k\in\N\).
        \item\label{rem:prox-tmm-distance-estimate} In particular, retaining only the last term in \eqref{eq:prox-tmm:lyapunov} gives
        \begin{equation*}
            \p{\forall k\in\N}\quad
            \mu\norm{z^k-x^{\star}}^2
            \leq
            \mathcal{V}_k^{\infty}.
        \end{equation*}
    \end{enumeratremark}
\end{remark}

Inspired by \eqref{eq:prox-item:one-step-lyapunov} in
\Cref{lem:prox-item-one-step-lyapunov}, we divide the scalar weights in
\eqref{eq:prox-item:one-step-lyapunov-remainder} by \(A_{k+1}\) and pass to
the coefficientwise limit using
\Cref{lem:prox-tmm-limiting-coefficients}, as in
\Cref{rem:prox-tmm-lyapunov:stationary-limit}. This gives the stationary
Lyapunov equality below.

\begin{lemma}[Lyapunov equality]\label{lem:prox-tmm-one-step-lyapunov}
    \sloppy
    Suppose that \Cref{ass:main} holds.
    Let~\(\p{\mathcal{V}_k^{\infty}}_{k\in\N}\) be as in
    \Cref{def:prox-tmm-lyapunov}.
    Then
    \begin{equation}\label{eq:prox-tmm:one-step-lyapunov}
        \p{\forall k\in\N}\quad
        \mathcal{V}_{k+1}^{\infty}
        =
        \p{1-\sqrt{q}}^{2}\mathcal{V}_k^{\infty}
        -\mathcal{R}_k^{\infty},
    \end{equation}
    where
    \begin{align}
        \p{\forall k\in\N}\quad
        \mathcal{R}_k^{\infty}
        ={}&{}
        \frac{1-\p{1-\sqrt{q}}^{2}}{2L}
        \norm{s_g^{k+1}-s_g^{\star}}^2
        +\frac{\p{1-\sqrt{q}}^{2}}{2L}
        \norm{s_g^k-s_g^{k+1}}^2 \notag\\
        {}&{}+
        \p{1-q}\p{1-\p{1-\sqrt{q}}^{2}}
        \mathcal{I}_f\p{x^{\star},y^k} \notag\\
        {}&{}+
        \p{1-q}\p{1-\sqrt{q}}^{2}
        \mathcal{I}_f\p{y^{k-1},y^k} \notag\\
        {}&{}+
        \p{1-\p{1-\sqrt{q}}^{2}}
        \mathcal{I}_g\p{x^{\star},z^{k+1},s_g^{k+1}} \notag\\
        {}&{}+
        \sqrt{q}\p{1-\sqrt{q}}^{3}
        \mathcal{I}_g\p{z^k,x^{\star},s_g^{\star}} \notag\\
        {}&{}+
        \sqrt{q}\p{2-\p{1-\sqrt{q}}^{2}}
        \mathcal{I}_g\p{z^{k+1},x^{\star},s_g^{\star}} \notag\\
        {}&{}+
        \sqrt{q}\p{1-\sqrt{q}}^{2}
        \mathcal{I}_g\p{z^{k+1},z^k,s_g^k}.
        \label{eq:prox-tmm:one-step-lyapunov-remainder}
    \end{align}
    Moreover, \(\mathcal{R}_k^{\infty}\in\reals_+\) for each \(k\in\N\).
\end{lemma}

\begin{proof}
    First, the fact that \(\mathcal{R}_k^{\infty}\in\reals_+\) for each
    \(k\in\N\) follows directly from
    \eqref{eq:prox-tmm:one-step-lyapunov-remainder}, \(q\in\p{0,1}\), the
    interpolation inequalities \eqref{eq:smooth_interpolation_inequality} and
    \eqref{eq:convex_interpolation_inequality:nonnegative}, and the inclusions
    \eqref{eq:prox-tmm:sampled-subgradient} and
    \eqref{eq:prox-item:optimal-subgradient}.

    Next, define
    \begin{equation}\label{eq:prox-tmm:one-step-lyapunov-slack-definition}
        \p{\forall k\in\N}\quad
        \mathcal{S}_k^{\infty}
        =
        \mathcal{V}_{k+1}^{\infty}
        -\p{1-\sqrt{q}}^{2}\mathcal{V}_k^{\infty}
        +\mathcal{R}_k^{\infty};
    \end{equation}
    it remains to show that \(\mathcal{S}_k^{\infty}=0\).

    Substituting \eqref{eq:prox-tmm:lyapunov} and
    \eqref{eq:prox-tmm:one-step-lyapunov-remainder} into
    \(\mathcal{S}_k^{\infty}\), and using
    \eqref{eq:smooth_interpolation_inequality}, the coefficients of the
    function values in the \(\mathcal{I}_f\)-terms are
    \begin{align*}
        f\p{y^{k-1}}:
        {}&{}
        -\p{1-q}\p{1-\sqrt{q}}^{2}
        +\p{1-q}\p{1-\sqrt{q}}^{2}=0, \\
        f\p{y^k}:
        {}&{}
        \p{1-q}
        -\p{1-q}\p{1-\p{1-\sqrt{q}}^{2}}
        -\p{1-q}\p{1-\sqrt{q}}^{2}=0, \\
        f\p{x^{\star}}:
        {}&{}
        -\p{1-q}
        +\p{1-q}\p{1-\sqrt{q}}^{2}
        +\p{1-q}\p{1-\p{1-\sqrt{q}}^{2}}=0.
    \end{align*}
    Similarly, using \eqref{eq:convex_interpolation_inequality}, the coefficients
    of the function values in the \(\mathcal{I}_g\)-terms are
    \begin{align*}
        g\p{x^{\star}}:
        {}&{}
        q-q\p{1-\sqrt{q}}^{2}
        +1-\p{1-\sqrt{q}}^{2}
        -\sqrt{q}\p{1-\sqrt{q}}^{3}
        -\sqrt{q}\p{2-\p{1-\sqrt{q}}^{2}}=0, \\
        g\p{z^{k+1}}:
        {}&{}
        -q-\p{1-\p{1-\sqrt{q}}^{2}}
        +\sqrt{q}\p{2-\p{1-\sqrt{q}}^{2}}
        +\sqrt{q}\p{1-\sqrt{q}}^{2}=0, \\
        g\p{z^k}:
        {}&{}
        q\p{1-\sqrt{q}}^{2}
        +\sqrt{q}\p{1-\sqrt{q}}^{3}
        -\sqrt{q}\p{1-\sqrt{q}}^{2}=0.
    \end{align*}
    Thus, all function values in \(\mathcal{S}_k^{\infty}\) cancel.

    Expanding the remaining parts of the interpolation inequalities gives
    \begin{align}
        \mathcal{S}_{k}^{\infty}
        ={}&{}
        \frac{1-\p{1-\sqrt{q}}^{2}}{2L}
        \norm{s_g^{k+1}-s_g^{\star}}^2
        +\frac{\p{1-\sqrt{q}}^{2}}{2L}
        \norm{s_g^k-s_g^{k+1}}^2 \notag\\
        {}&{}
        -\p{1-q}\p{1-\p{1-\sqrt{q}}^{2}}\p{
            \inner{\nabla f\p{y^k}}{x^{\star}-y^k}
            +\frac{\mu}{2}\norm{x^{\star}-y^k}^{2}
        } \notag\\
        {}&{}
        -\frac{\p{1-q}\p{1-\p{1-\sqrt{q}}^{2}}}{2\p{L-\mu}}
        \norm{\nabla f\p{x^{\star}}-\nabla f\p{y^k}-\mu\p{x^{\star}-y^k}}^2 \notag\\
        {}&{}
        -\p{1-q}\p{1-\sqrt{q}}^{2}\p{
            \inner{\nabla f\p{y^k}}{y^{k-1}-y^k}
            +\frac{\mu}{2}\norm{y^{k-1}-y^k}^{2}
        } \notag\\
        {}&{}
        -\frac{\p{1-q}\p{1-\sqrt{q}}^{2}}{2\p{L-\mu}}
        \norm{\nabla f\p{y^{k-1}}-\nabla f\p{y^k}-\mu\p{y^{k-1}-y^k}}^2 \notag\\
        {}&{}
        -\p{1-q}\p{
            \inner{\nabla f\p{x^{\star}}}{y^k-x^{\star}}
            +\frac{\mu}{2}\norm{y^k-x^{\star}}^{2}
        } \notag\\
        {}&{}
        -\frac{\p{1-q}}{2\p{L-\mu}}
        \norm{\nabla f\p{y^k}-\nabla f\p{x^{\star}}-\mu\p{y^k-x^{\star}}}^2 \notag\\
        {}&{}
        +\p{1-q}\p{1-\sqrt{q}}^{2}\p{
            \inner{\nabla f\p{x^{\star}}}{y^{k-1}-x^{\star}}
            +\frac{\mu}{2}\norm{y^{k-1}-x^{\star}}^{2}
        } \notag\\
        {}&{}
        +\frac{\p{1-q}\p{1-\sqrt{q}}^{2}}{2\p{L-\mu}}
        \norm{\nabla f\p{y^{k-1}}-\nabla f\p{x^{\star}}-\mu\p{y^{k-1}-x^{\star}}}^2 \notag\\
        {}&{}
        -\p{1+q-\p{1-\sqrt{q}}^{2}}\inner{s_g^{k+1}}{x^{\star}-z^{k+1}}
        +q\p{1-\sqrt{q}}^{2}\inner{s_g^k}{x^{\star}-z^k} \notag\\
        {}&{}
        -\sqrt{q}\p{1-\sqrt{q}}^{3}
        \inner{s_g^{\star}}{z^k-x^{\star}} \notag\\
        {}&{}
        -\sqrt{q}\p{2-\p{1-\sqrt{q}}^{2}}
        \inner{s_g^{\star}}{z^{k+1}-x^{\star}} \notag\\
        {}&{}
        -\sqrt{q}\p{1-\sqrt{q}}^{2}\inner{s_g^k}{z^{k+1}-z^k}
        -\frac{\p{1-\sqrt{q}}^{2}}{2L}\norm{s_g^k-s_g^{\star}}^2
        +\frac{1}{2L}\norm{s_g^{k+1}-s_g^{\star}}^2 \notag\\
        {}&{}
        +\mu\norm{z^{k+1}-x^{\star}}^2
        -\mu\p{1-\sqrt{q}}^{2}\norm{z^k-x^{\star}}^2.
        \label{eq:prox-tmm:one-step-lyapunov-raw-slack}
    \end{align}
    Using \(q=\mu/L\) and the update rules of \refProxTMM{}, together with
    \eqref{eq:prox-tmm:sampled-subgradient} and
    \eqref{eq:prox-item:optimal-subgradient}, we obtain from
    \eqref{eq:prox-tmm:one-step-lyapunov-raw-slack} after direct
    simplification\footnote{\raggedright Symbolic verification notebooks are available at
    \href{https://github.com/PerformanceEstimation/ProxITEM}{\nolinkurl{https://github.com/PerformanceEstimation/}}\allowbreak\href{https://github.com/PerformanceEstimation/ProxITEM}{\mbox{\nolinkurl{ProxITEM}}}.\par}
    \begin{equation}
        \p{\forall k\in\N}\quad
        \mathcal{S}_{k}^{\infty}=0.
        \label{eq:prox-tmm:one-step-lyapunov-slack-zero}
    \end{equation}
    Combining \eqref{eq:prox-tmm:one-step-lyapunov-slack-definition} with
    \eqref{eq:prox-tmm:one-step-lyapunov-slack-zero} gives
    \eqref{eq:prox-tmm:one-step-lyapunov}.
\end{proof}


\section{Proof of \texorpdfstring{\Cref{thm:main:prox_item}}{the main Prox-ITEM theorem}}\label{sec:proof_main_prox_item}
The case \(k=0\) holds trivially, since \eqref{eq:main:prox_item} reduces~to
\begin{align*}
    \norm{x^0-x^{\star}}^{2}\leq \norm{x^0-x^{\star}}^{2},
\end{align*}
where we have used that \(A_0 = 0\) and \(x^0 = z^0\).

Assume now that \(k\geq1\). By \Cref{lem:prox-item-one-step-lyapunov}, induction gives
\begin{equation}\label{eq:proof-main-prox-item:energy-recursion}
    \mathcal{V}_{k}
    \leq
    \mathcal{V}_0.
\end{equation}
Combining \Cref{rem:prox-item-initial-lyapunov,rem:prox-item-distance-estimate} with \eqref{eq:proof-main-prox-item:energy-recursion} gives
\begin{equation}\label{eq:proof-main-prox-item:distance-bound}
    \p{L+\mu A_k}\norm{z^k-x^{\star}}^2
    \leq
    \mathcal{V}_{k}
    \leq
    \mathcal{V}_0
    =
    L\norm{x^0-x^{\star}}^2.
\end{equation}
Since \(q=\mu/L\), we have \(L+\mu A_k=L\p{1+qA_k}\).
Thus, dividing \eqref{eq:proof-main-prox-item:distance-bound} by \(L\p{1+qA_k}\) proves \eqref{eq:main:prox_item}.\nobreak\hfill\quad\(\square\)


\section{Proof of \texorpdfstring{\Cref{thm:main:prox_item_is_optimal}}{the optimality of Prox-ITEM}}\label{sec:proof_main_prox_item_is_optimal}

Let \(\mathsf{A}_{\text{\ProxITEM},k}\) denote \refProxITEM{} run for \(k\in\naturals\)
iterations with input \(x^0=u^0\) and output \(z^k\). Then
\(\mathsf{A}_{\text{\ProxITEM},k}\in\mathcal{A}_{k,\mu,L}\), as follows from
the update rule \eqref{eq:span-based-proximal-method-update} under the
identification
\[
    \begin{gathered}
        m=2k,\qquad u^k=z^k,\\[-0.35\baselineskip]
        \p{\forall \ell\in\llbracket0,k-1\rrbracket},\\[-0.35\baselineskip]
        \begin{alignedat}{4}
            \theta_{2\ell} &= 1,\qquad&
            \theta_{2\ell+1} &= 0,\qquad&
            v^{2\ell} &= y^\ell,\qquad&
            v^{2\ell+1} &= \bar z^{\ell+1},
            \\
            s^{2\ell} &= \nabla f\p{y^\ell},\qquad&
            s^{2\ell+1} &= \frac{L}{\delta_\ell}\p{\bar z^{\ell+1}-z^{\ell+1}},\qquad&
            \gamma_{2\ell+1} &= \frac{\delta_\ell}{L}.
        \end{alignedat}
    \end{gathered}
\]

By \Cref{thm:main:prox_item}, \(\mathsf{A}_{\text{\ProxITEM},k}\) satisfies
\begin{equation}\label{eq:proof-main-prox-item-is-optimal:upper-bound}
    \mathcal{R}_{\mu,L}\p{\mathsf{A}_{\text{\ProxITEM},k}}
    \leq
    \frac{1}{1+qA_k}.
\end{equation}
This proves that the left-hand side of \eqref{eq:main:prox_item_is_optimal} is
no larger than the right-hand side.

For the reverse inequality, restrict the supremum in
\eqref{eq:fixed-step-proximal-method-risk} to instances with \(g=0\). Then all
proximal residuals are zero, and \(\mathcal{A}_{k,\mu,L}\) reduces to a
black-box first-order method using \(k\) gradient evaluations of a function in
\(\mathcal{F}_{\mu,L}\p{\calH}\). The~lower~bound~for smooth and strongly convex functions from
\cite[Corollary~4]{drori2021exact}, together with
\cite[Section~2.3]{taylor2023optimal_gradient}, gives,\nolinebreak[4]\ for each
\(\mathsf A\in\mathcal{A}_{k,\mu,L}\),
\begin{equation}\label{eq:proof-main-prox-item-is-optimal:lower-bound}
    \mathcal{R}_{\mu,L}\p{\mathsf A}
    \geq
    \frac{1}{1+qA_k}.
\end{equation}
Taking the infimum over \(\mathsf A\in\mathcal{A}_{k,\mu,L}\) in
\eqref{eq:proof-main-prox-item-is-optimal:lower-bound} and combining it with
\eqref{eq:proof-main-prox-item-is-optimal:upper-bound} proves
\eqref{eq:main:prox_item_is_optimal}.\nobreak\hfill\quad\(\square\)


\section{Proof of \texorpdfstring{\Cref{thm:main:prox_tmm}}{the main Prox-TMM theorem}}\label{sec:proof_main_prox_tmm}

By \Cref{lem:prox-tmm-one-step-lyapunov}, induction gives
\begin{equation}\label{eq:proof-main-prox-tmm:energy-bound}
    \p{\forall k\in\N}\quad
    \mathcal{V}_k^{\infty}
    \leq
    \p{1-\sqrt{q}}^{2\p{k-1}}
    \mathcal{V}_1^{\infty}.
\end{equation}
Combining \eqref{eq:proof-main-prox-tmm:energy-bound}
with \Cref{rem:prox-tmm-distance-estimate} gives
\begin{equation}\label{eq:proof-main-prox-tmm:distance-bound}
    \p{\forall k\in\N}\quad
    \mu\norm{z^k-x^{\star}}^2
    \leq
    \p{1-\sqrt{q}}^{2\p{k-1}}
    \mathcal{V}_1^{\infty}.
\end{equation}
By the definition of \(c\) in \Cref{thm:main:prox_tmm} and
\eqref{eq:prox-tmm:lyapunov}, \(c=\mu^{-1}\mathcal{V}_1^{\infty}\).
Dividing \eqref{eq:proof-main-prox-tmm:distance-bound} by~\(\mu\)~gives
\begin{equation*}
    \p{\forall k\in\N}\quad
    \norm{z^k-x^{\star}}^2
    \leq
    c\p{1-\sqrt{q}}^{2\p{k-1}},
\end{equation*}
as claimed.\nobreak\hfill\quad\(\square\)


\section{Conclusion}\label{sec:conclusion}

{\sloppy
We established that the exact last-iterate distance guarantee of ITEM for smooth, strongly convex minimization extends to the composite forward-backward oracle model.
The proposed method, \ProxITEM{}, attains the factor \(\p{1+\allowbreak qA_k}^{-1}\) for the squared distance to the minimizer, and this factor is optimal among span-based methods using \(k\) gradient-\allowbreak oracle calls for \(f\) and an arbitrary number of proximal-\allowbreak oracle calls for \(g\).
We also identified the stationary limit, \ProxTMM{}, and proved that it preserves the linear distance rate of the triple momentum method in the composite setting.
\par}

Beyond the algorithms themselves, the analysis developed here yields concise, human-verifiable proofs using Lyapunov functions.
The methods were first identified by solving long-horizon performance-estimation problems. These problems were formulated as semidefinite programs whose size grows with the time horizon \(k\); see, e.g.,~\cite{goujaud2024pepit} and the accompanying \href{https://github.com/PerformanceEstimation/ProxITEM}{GitHub repository}.
Converting the resulting SDP certificates into the compact proofs based on Lyapunov functions presented here required substantial simplification of our initial proofs, which were closer in spirit to, e.g.,~\cite{bok_altschuler_2025_composite_reduction,jang2025optista,kim_fessler_2016_ogm,kim_fessler_2018_generalizing_ogm}.
A natural next step is to understand the structure of such optimal certificates more systematically, with the goal of identifying more interpretable Lyapunov functions that can be carried beyond the forward-backward model considered here.
This perspective may offer a route to designing optimal methods, together with short proofs of their worst-case guarantees, in broader composite settings such as primal-dual, multi-operator, and linearly composite problems.


\section*{Declarations}
\phantomsection\label{sec:declarations}

\subsection*{Funding}
\phantomsection\label{subsec:funding}
M. Upadhyaya, D. B. Thomsen, and A. B. Taylor are supported by the European Union (ERC grant CASPER 101162889). This work was also partially funded by the French government, through the Agence Nationale de la Recherche, as part of the ``France 2030'' program under grant references ANR-23-IACL-0008 (``PR[AI]RIE-PSAI''), ANR-23-PEIA-0005 (``REDEEM''), and ANR-23-IACL-0005 (``Hi! PARIS''). The views and opinions expressed are those of the authors only and do not necessarily reflect those of the funding agencies or granting authorities, which cannot be held responsible for them.

\subsection*{Use of generative AI}
\phantomsection\label{subsec:use_of_generative_ai}
{\raggedright
The authors used large-language-model (LLM) tools, including ChatGPT and Codex, to assist with writing, copyediting, and exploratory coding.
These tools were also used in connection with exploratory full-\allowbreak horizon performance-\allowbreak estimation computations that helped generate initial candidate methods and proofs. The final Lyapunov proofs presented in this paper were derived and verified by the authors and substantially simplify the initial LLM proposals. The authors take full responsibility for all content.
\par}

\subsection*{Competing interests}
\phantomsection\label{subsec:conflict_of_interest}
The authors declare no relevant financial or non-financial interests to disclose.

\subsection*{Code availability}
\phantomsection\label{subsec:code_availability}
Numerical and symbolic verification notebooks accompanying this paper are available at
\href{https://github.com/PerformanceEstimation/ProxITEM}{\nolinkurl{https://github.com/PerformanceEstimation/}}%
\ifsiopt\\\else\allowbreak\fi
\href{https://github.com/PerformanceEstimation/ProxITEM}{\mbox{\nolinkurl{ProxITEM}}}.

\subsection*{Data availability}
\phantomsection\label{subsec:data_availability}
No datasets were generated or analyzed for this paper.

\bibliographystyle{abbrvnat}
\bibliography{bib}

\end{document}